\documentclass[]{article}
\usepackage{indentfirst}
\usepackage{amsmath}
\usepackage{authblk}
\usepackage{geometry}
\usepackage{accents}
\usepackage{statex}
\usepackage[francais,english]{babel}

\newtheorem{proposition}{Proposition}

\newenvironment{proof}[1][Proof]{\begin{trivlist}
		\item[\hskip \labelsep {\bfseries #1}]}{\end{trivlist}}
\newenvironment{definition}[1][Definition]{\begin{trivlist}
		\item[\hskip \labelsep {\bfseries #1}]}{\end{trivlist}}

\title{(non)-automaticity of completely multiplicative sequences having negligible many non-trivial prime factors}
\author{Shuo LI\\
CNRS, Institut de Math\'ematiques de Jussieu-PRG\\
Universit\'e Pierre et Marie Curie, Case 247\\
4 Place Jussieu\\
F-75252 Paris Cedex 05 (France)\\
\normalsize{shuo.li@imj-prg.fr}\\}
\date {}

\begin{document}
	
\maketitle

\begin{abstract} 
In this article we consider the completely multiplicative sequences $(a_n)_{n \in \mathbf{N}}$ defined on a field $\mathbf{K}$ and satisfying $$\sum_{p| p \leq n, a_p \neq 1, p \in \mathbf{P}}\frac{1}{p}<\infty,$$ where $\mathbf{P}$ is the set of prime numbers. We prove that if such sequences are automatic then they cannot have infinitely many prime numbers $p$ such that  $a_{p}\neq 1$. Using this fact, we prove that if a completely multiplicative sequence $(a_n)_{n \in \mathbf{N}}$, vanishing or not, can be written in the form $a_n=b_n\chi_n$ such that $(b_n)_{n \in \mathbf{N}}$ is a non  ultimately periodic, completely multiplicative automatic sequence satisfying the above condition, and $(\chi_n)_{n \in \mathbf{N}}$ is a Dirichlet character or a constant sequence, then there exists only one prime number $p$ such that $b_p \neq 1$ or $0$.
\end{abstract}

\section{Introduction}

The propose of this article is to study the automaticity of some completely multiplicative sequences, which will be denoted as CMS. In article [5], the author proves that if a non vanishing CMS is automatic then it is almost periodic (defined in [5]). In article [1], the author gives a formal expression to all non vanishing automatic CMSs and also some examples in the vanishing case. In article [3], the author studies the CMSs taking values on a general field who have finitely many prime numbers such that $a_p \neq 1$, she proves that such CMSs have a complexity $p_a(n)=O(n^k)$ where $k=\#\left\{p| p\in \mathbf{P}, a_p \neq 1,0\right\}$. In this article we consider the CMSs $(a_n)_{n \in \mathbf{N}}$ who satisfy the condition $\mathcal{C}_1: \sum_{p| p \leq n, a_p \neq 1, p \in \mathbf{P}}\frac{1}{p}<\infty$. We prove that such sequences cannot satisfy the condition $\mathcal{C}_2: \#\left\{p_i|a_{p_i}\neq 1\right\}=\infty$. This fact deduce that if a CMS $(a_n)_{n \in \mathbf{N}}$, vanishing or not, can be written in the form $a_n=b_n\chi_n$ with $(b_n)_{n \in \mathbf{N}}$ a non  ultimately periodic CMS satisfying the condition $\mathcal{C}_1$, and $(\chi_n)_{n \in \mathbf{N}}$ a Dirichlet character or a constant sequence, then there exists only one prime number $p$ such that $b_p \neq 1$ or $0$.

\section{Definitions and notations}

Here we declare some definitions and notations used in this article. We say a word of a sequences to be a finitely many long string of the sequence, we denote by $\overline{w}_l$ a word of length $l$. Let $(a_n)_{n \in \mathbf{N}}$ be a CMS, we say $a_p$ is a prime factor of $(a_n)_{n \in \mathbf{N}}$ if p is a prime number and $a_p \neq 1$, and $a_p$ is a non trivial prime factor if $a_p \neq 0, 1$. Let $(a_n)_{n \in \mathbf{N}}$ and $(b_n)_{n \in \mathbf{N}}$ be two CMSs we say $(a_n)_{n \in \mathbf{N}} < (b_n)_{n \in \mathbf{N}}$ if all prime factors of $(a_n)_{n \in \mathbf{N}}$ are prime factors of $(b_n)_{n \in \mathbf{N}}$. We say a sequence $(a_n)_{n \in \mathbf{N}}$ is generated by $a_{p_1}, a_{p_2},...$ if and only if $a_{p_1}, a_{p_2},...$ are the whole prime factors of the sequence.

We note $f(n) = O(g(n))$: f is bounded above by g (up to constant factor);

$f(n) = \Theta(g(n))$: f is bounded both above and below (up to constant factors) by g.

We recall the definitions of the automaticity of a sequence.  

\begin{definition}
Let $(a_n)_{n \in \mathbf{N}}$ be a infinite sequence and $k \leq 2$ an integer, we say this sequence is $k$-automatic if there is a finite set of sequences containing $(a_n)_{n \in \mathbf{N}}$ and closed under the map
$$a_n \rightarrow a_{kn+i}, i=0,1,...k-1$$
\end{definition}

\section{Automaticity}

\begin{proposition}
Let $(a_n)_{n \in \mathbf{N}}$ be a CMS defined on the set $G=\left\{\zeta^r| r \in \mathbf{N} \right\}$ included in $\mathbf{F}$ where $\zeta$ is a $k$-th root of unity and it has finitely many prime factors $a_{p_1}, a_{p_2}, ... a_{p_u}$, then for every  sub-set $G_1=\left\{{\zeta^{'}}^r| r \in \mathbf{N},\zeta^{'} \in G/\left\{1\right\} \right\}$, there exists a word $\overline{w}_u$ appearing periodically in the sequence $(a_n)_{n \in \mathbf{N}}$ such that none of its letters appears in the set $G_1/\left\{1\right\}$. What is more, the period does not have any other prime factors than $p_1, p_2, ..., p_u$.

\end{proposition}

\begin{proof}

We prove it by recurrence, for $u=1$,  the above statement is trivial. It is easy to check that the sequence $(a_{np_1^{k+1}+p_1^k})_{n \in \mathbf{N}}$ is a constant sequence of $1$ and the period is $p_1^{k+1}$. 

For an arbitrary sub-group $G_1$ of $G$ defined as above, supposing the statement is true for some $n_0$, that is to say, there exists a word $\overline{w}_{n_0}$ who does not contain any letters in the set $G_1/\left\{1\right\}$ can be extracted periodically, and the first letter of such word lies repeatedly on the sequence $(a_{m_{n_0}n+l_{n_0}})_{n \in \mathbf{N}}$ where $m_{n_0}=\prod_{j=1}^{n_0}p_j^{c_j}$ for some $c_j \in \mathbf{N}^{+}$. Let us consider the case $u=n_0+1$, we firstly consider the the sequence $(a_n^{'})_{n \in \mathbf{N}}$ defined as $a_n^{'} = \frac{a_n}{v_{p_{n_0+1}}(n)}$, a sequence having $n_0$ prime factors. Using the recurrence hypothesis we get a word $\overline{w}_{n_0}$ appearing periodically and the first letter of this word lies on $(a^{'}_{m_{n_0}n+l_{n_0}})_{n \in \mathbf{N}}$ defined as above. We can extract this sequence once more to the form $(a_{m_{n_0^{'}}n+l_{n_0}})_{n \in \mathbf{N}}$ such that $m_{n_0^{'}}=m_{n_0}\prod_{j=1}^{n_0}p_j^{d_j}$ for some $d_j \in \mathbf{N}^{+}$ and $v_{p_j}(m_{n_0^{'}}n+l_{n_0}+n_0)=v_{p_j}(l_{n_0}+n_0)$ for all $j \leq n_0$. In this case the sequence $(a^{'}_{m_{n_0^{'}}n+l_{n_0}+n_0})_{n \in \mathbf{N}}$ is a constant sequence, say all letters equal $C$.

Here we consider the sequence $(a_n)_{n \in \mathbf{N}}$, we want to find a sequence $N(n)=np_{n_0+1}^k+r$ for some $k$ and $r \in \mathbf{N}$ such that $p_{n_0+1}| m_{n_0^{'}}N(n)+l_{n_0}+n_0$ for all $n \in \mathbf{N}$ and $(a_{m_{n_0^{'}}N(n)+l_{n_0}+n_0})_{n \in \mathbf{N}}$ to be a constant sequence, what is more, the constant is not in $G_1/\left\{1\right\}$. We discuss the cases as following:

If $C \in G_1/\left\{1\right\}$ while $a_{p_{n_0+1}} \not\in G_1$, as while as the case $C  \not\in G_1/\left\{1\right\}$ while $a_{p_{n_0+1}}\in G_1$, we can find a sequence $N(n)_n \in \mathbf{N}$ satisfying
 $$m_{n_0^{'}}N(n) \equiv -l_{n_0}-n_0\mod p_{n_0+1}$$ $$m_{n_0^{'}}N(n) \not\equiv -l_{n_0}-n_0\mod p_{n_0+1}^2$$  for all $n$ to guarantee the sequence $(a_{m_{n_0^{'}}n+l_{n_0}+n_0})_{n \in \mathbf{N}}$ to satisfy the above hypothesis. Take for example $N(n)=np_{n_0+1}^{2}-l_{n_0}-n_0+kp_{n_0+1}$ for some $k$.

If $C \in G_1/\left\{1\right\}$ while $a_{p_{n_0+1}} \in G_1$, then there exists a $s_1 \in \mathbf{N}$ such that $Ca_{p_{n_0+1}}^{s_1}=1$. we want a sequence $N(n)_n \in \mathbf{N}$ satisfying
 $$m_{n_0^{'}}N(n) \equiv -l_{n_0}-n_0\mod p_{n_0+1}^{s_1}$$ $$m_{n_0^{'}}N(n) \not\equiv -l_{n_0}-n_0\mod p_{n_0+1}^{s_1+1}$$  for all $n$. Take for example $N(n)=np_{n_0+1}^{s_1+1}-l_{n_0}-n_0+kp_{n_0+1}^{s_1}$ for some $k$.

If $C \not \in G_1/\left\{1\right\}$ while $a_{p_{n_0+1}} \not \in G_1$, then there exists a $s_2 \in \mathbf{N}$ such that $Ca_{p_{n_0+1}}^{s_2} \not \in G_1/\left\{1\right\}$. we want a sequence $N(n)_n \in \mathbf{N}$ satisfying
 $$m_{n_0^{'}}N(n) \equiv -l_{n_0}-n_0\mod p_{n_0+1}^{s_2}$$ $$m_{n_0^{'}}N(n) \not\equiv -l_{n_0}-n_0\mod p_{n_0+1}^{s_2+1}$$  for all $n$. Take for example $N(n)=np_{n_0+1}^{s_2+1}-l_{n_0}-n_0+kp_{n_0+1}^{s_2}$ for some $k$.

The above argument shows that there exist $m_{n_0+1} \in \mathbf{N}, l_{n_0+1} \in \mathbf{N}$ such that for all $n \in \mathbf{N}$, the word $\overline{\frac{a_{m_{n_0+1}n+l_{n_0+1}}}{a_{p_{n_0+1}}^{v_{p_{n_0+1}}(m_{n_0+1}n+l_{n_0+1})}}\frac{a_{m_{n_0+1}n+l_{n_0+1}+1}}{a_{p_{n_0+1}}^{v_{p_{n_0+1}}(m_{n_0+1}n+l_{n_0+1}+1)}}...\frac{a_{m_{n_0+1}n+l_{n_0+1}+n_0-1}}{a_{p_{n_0+1}}^{v_{p_{n_0+1}}(m_{n_0+1}n+l_{n_0+1}+n_0-1)}}}$ is constant  and none of its letters in $G_1/\left\{1\right\}$, and $(a_{m_{n_0+1}n+l_{n_0+1}+n_0})_{n \in \mathbf{N}}$ is a constant sequence not in $G_1/\left\{1\right\}$. We remark that the prime number $p_{n_0+1}$ satisfies $p_{n_0+1} > n_0+1$ and $p_{n_0+1}|m_{n_0+1}n+l_{n_0+1}+n_0$ because of the construction. These properties imply that for all $0 \leq j \leq n_0-1$, $p_{n_0+1} \nmid m_{n_0+1}n+l_{n_0+1}+j$. So we conclude that for all $n \in \mathbf{N}$ and $0\leq j\leq n_0-1$, $v_{p_{n_0+1}}(m_{n_0+1}n+l_{n_0+1}+j)=0$, that means the word $\overline{a_{m_{n_0+1}n+l_{n_0+1}}a_{m_{n_0+1}n+l_{n_0+1}+1}...a_{m_{n_0+1}n+l_{n_0+1}+n_0}}$ is a constant word of length $n_0+1$ and none of its letters in $G_1/\left\{1\right\}$, what is more $m_{n_0+1}$ does not have any other prime factors other than $p_1, p_2 ... p_{n_0}$.
\end{proof}

\begin{proposition}
Let $(a_n)_{n \in \mathbf{N}}$ be a CMS defined on a finite set $G$ of $F$ satisfying conditions $\mathcal{C}_1$ and $\mathcal{C}_2$, and let $(a_n^{'})_{n \in \mathbf{N}}$ be another CMS generated by the first $r$ prime factors of $(a_n)_{n \in \mathbf{N}}$, say $a_{p_1}, a_{p_2} ,..., a_{p_r}$. If there is a word $\overline{w}_r$ appears periodically in $(a_n^{'})_{n \in \mathbf{N}}$, and the periodic does not have any other prime factors than $p_1,p_2,...,p_r$, then this word appears at least once in $(a_n)_{n \in \mathbf{N}}$.
\end{proposition}

\begin{proof}

Let us denote by $p_1, p_2...$ the sequence of prime numbers such that $a_{p_i} \neq 1$. Supposing the first letter of the word $\overline{w}_r$ lies on the sequence $(a_{m_rn+l_r}^{'})_{n \in \mathbf{N}}$ for some $m_{r} \in \mathbf{N}, l_{r} \in \mathbf{N}$, by hypothesis, $m_r$ does not has any other prime factors than $p_1,p_2,...,p_r$.  So the total number of such word in the sequence $(a_{n})_{n \in \mathbf{N}}$ can be bounded by the inequality:
\begin{equation} 
\#\left\{a_k| k \leq n, \overline{a_k,a_{k+1},...,a_{k+r-1}}=\overline{w}_{r}\right\} \geq\#\left\{a_k| k \leq n, k=m_rk^{'}+l_r, k^{'} \in \mathbf{N}; k+j \nmid p_i,  \forall 0\leq j \leq r-1,  \forall i > r\right\} \\ 
\end{equation}
 Let us consider the sequence defined as $N(i)=\prod_{j=1}^ip_{r+j}$,  we have
\begin{equation} 
\begin{aligned} 
\#\left\{a_k| k \leq N(i)m_r+l_r, k=m_rk^{'}+l_r, k^{'} \in \mathbf{N}; k+j \nmid p_k,  \forall 0\leq j \leq r-1, r < k \leq r +i \right\}=\prod_{j=1}^{i}(p_{r+j}-r) 
\end{aligned}
\end{equation}
 
This equality holds because of Chinese reminder theorem, and the fact that $p_{r+i} \nmid m_r$ and $p_{r+j} > r$ for all $i \geq 1$.

So we have
\begin{equation} 
\begin{aligned} 
&\#\left\{a_k| k \leq N(i)m_r+l_r, k=m_rk^{'}+l_r, k^{'} \in \mathbf{N}; k+j \nmid p_i,  \forall 0\leq j \leq r-1,  \forall i > r\right\}\\
>&\#\left\{a_k| k \leq N(i)m_r+l_r, k=m_rk^{'}+l_r, k^{'} \in \mathbf{N}; k+j \nmid p_k,  \forall 0\leq j \leq r-1, r < k \leq r +i \right\}\\
&-\#\left\{a_k| k \leq N(i)m_r+l_r, k=m_rk^{'}+l_r, k^{'} \in \mathbf{N}; k+j \nmid p_k,  \forall 0\leq j \leq r-1, k > r +i \right\}\\
>&\#\left\{a_k| k \leq N(i)m_r+l_r, k=m_rk^{'}+l_r, k^{'} \in \mathbf{N}; k+j \nmid p_k,  \forall 0\leq j \leq r-1, r < k \leq r +i \right\}\\
&-\sum_{k > r +i}\#\left\{a_k| k \leq N(i)m_r+l_r, k=m_rk^{'}+l_r, k^{'} \in \mathbf{N}; k+j \nmid p_k,  \forall 0\leq j \leq r-1\right\}\\
>&\prod_{j=1}^{i}(p_{r+j}-r) -\sum_{k > r +i}\frac{N(i)}{p_k}-O(\log(N(i)))
\end{aligned}
\end{equation}

However, 

\begin{equation} 
\prod_{j=1}^{i}(p_{r+j}-r)=\prod_{j=1}^{i}\frac{p_{r+j}-r}{p_{r+j}}N(i) \geq \prod_{j=1}^{\infty}\frac{p_{r+j}-r}{p_{r+j}}N(i)
\end{equation}

as $\prod_{j=1}^{\infty}\frac{p_{r+j}-r}{p_{r+j}}=\exp(\sum_{j=1}^{\infty}\log(\frac{p_{r+j}-r}{p_{r+j}}))=\exp(-\Theta(\sum_{j=1}^{\infty}\frac{r}{p_{r+j}}))$, because of $\mathbf{C_1}$, the above equation does not converge to $0$, We conclude that there exists $0<c<1$ such that $\prod_{j=1}^{i}(p_{r+j}-r) > cN(i)$.

On the other hand, we remark that for all $k > r +i$, $p_{k}^i>\prod_{j=1}^{i}p_{r+j}=N(i)$, so $p_{k}>N(i)^{\frac{1}{i}}$

\begin{equation} 
\sum_{k > r +i}\frac{N(i)}{p_k} <N(i) \sum_{N(i)^{\frac{1}{i}}< p < N(i)} \frac{1}{p}
\end{equation}

However, 

\begin{equation} 
N(i)^{\frac{1}{i}}=(\prod_{j=1}^{i}p_{r+j})^{\frac{1}{i}} \geq \frac{i}{\sum_{j=1}^{i}\frac{1}{p_{r+j}}} >  \frac{i}{\sum_{j=1}^{i}\frac{1}{q_{j}}}
\end{equation}

where $q_j$ is the $j$-th prime number in $\mathbf{N}$. For any $x \in \mathbf{N}$, $\#\left\{p_i| p_i \leq x\right\}=\Theta(\log(x))$ and $\sum_{p_i \leq x} \frac{1}{p_i}=\Theta(\log\log(x))$, so $N(i)^{\frac{1}{i}}$  tends to infinity when $i$ tends to infinity,  because of $\mathbf{C_1}$, we can conclude there exists some $i_0 \in \mathbf{N}$ such that for all $i > i_0$, $\sum_{N(i)^{\frac{1}{i}}< p < N(i)} \frac{1}{p} < \frac{1}{2}c$.

To conclude,  for all $i > i_0$,

\begin{equation} 
\begin{aligned} 
&\#\left\{a_k| k \leq N(i)m_r+l_r, k=m_rk^{'}+l_r, k^{'} \in \mathbf{N}; k+j \nmid p_i,  \forall 0\leq j \leq r-1,  \forall i > r\right\}\\
>&\prod_{j=1}^{i}(p_{r+j}-r) -\sum_{k > r +i}\frac{N(i)}{p_k}-O(\log(N(i)))\\
>&cN(i)- \frac{1}{2}cN(i)-O(\log(N(i))
\end{aligned}
\end{equation}

When $i$ tends to infinity,  the set $\#\left\{a_k| k \leq n, \overline{a_k,a_{k+1},...,a_{k+r-1}}=\overline{w}_{r}\right\}$ is not empty.

\end{proof}

\begin{proposition}
Let $(a_n)_{n \in \mathbf{N}}$ be a $p$-automatic CMS, vanishing or not. If it can be written in the form $a_n=b_n\chi_n$ with $(b_n)_{n \in \mathbf{N}}$ a non ultimately periodic CMS satisfying the condition $\mathcal{C}_1$, and $(\chi_n)_{n \in \mathbf{N}}$ a Dirichlet character or a constant sequence, then there exists only one prime number $p$ such that $b_p \neq 1$ or $0$
\end{proposition}

\begin{proof}
Let us consider the sequence $(b_n)_{n \in \mathbf{N}}$, who is also a $p$-automatic CMS, if this sequence satisfies the condition $\mathbf{C}_1$ and $\mathbf{C}_2$, then we can complete this sequence to $(b_n^{'})_{n \in \mathbf{N}}$ in such a way that $b_n^{'}=b_n$ if $b_n \neq 0$, and $b_p^{'}=1$ for all prime numbers $p$ such that $b_p=0$ and all other $b_m^{'}$ when $b_m=0$ by multiplicity. In this way we obtain a completely multiplicative sequence satisfying the hypotheses of proposition 2. Let $G_1$ be a sub-group defined as in proposition 1. Let us consider the automaton generating the sequence $(b_n)_{n \in \mathbf{N}}$, if it has $q$ states, the above proposition proves there exists a word of length $p^{2q!}$, say $\overline{w}_{p^{2q!}}$, of the sequence $(b_n^{'})_{n \in \mathbf{N}}$ such that none of its letters appears in $G_1/\left\{1\right\}$. Then we can extract a sub-word $\overline{w'}_{p^{q!}}$ contained in $\overline{w}_{p^{2q!}}$ and of the form $\overline{b_{up^{q!}}^{'}b_{up^{q!}+2}^{'}...b_{(u+1)p^{q!}-1}^{'}}$ for some $u \in \mathbf{N}$. Because of the construction of the sequence $(b_n^{'})_{n \in \mathbf{N}}$, we can conclude that the word $\overline{b_{up^{q!}}b_{up^{q!}+2}...b_{(u+1)p^{q!}-1}}$ does not have any letters in $G_1/\left\{1\right\}$. However, in article [1](Lemma 3 and Theory 1), the author proves that in a automaton, every state which can be reached from a specific state, say $s$, with $q!$ steps, can be reached with $yq!$ steps for every $y \geq1$; and inversely, if a state can be reached with $yq!$ steps for some $y \geq 1$, then it can already be reached with $q!$ steps. so we can conclude that for every $y \geq 1$ and every $0 \leq m \leq p^{yq!}-1$, $a_{up^{yq!}+m }\not \in G_1/\left\{1\right\}$.

On the other hand, the article [4] proves that for every finite Abelian group (Theory 3.10) or a semi-group(Theory 7.3) $G$, $g \in G$ and $G$-multiplicative $f$ the sequence $f^{-1}(g)=\left\{n: f(n)=g\right\}$ has a non zero natural density. If we denote by $\left\{l_1,l_2,...l_i\right\}$ the set of letters appearing in $(a_n)_{n \in \mathbf{N}}$ and $k_1,k_2,...k_i$ their densities associated. Then we have for every $1 \leq j \leq i$,
$$\lim_{y \to \infty}\frac{1}{p^{yq!}}\#\left\{a_s = l_k|up^{yq!} \leq s < (u+1)p^{yq!} \right\} =k_j.$$

The above fact shows that all $l_r \in G_1/\left\{1\right\}$ have a 0 density, contradiction to theory 7.3 of article 2. So we deduce that the sequence $(b_n)_{n \in \mathbf{N}}$ must have a finitely many prime numbers satisfying $b_p \neq 1$. However, the corollary 2 of article [3] proves in this case,  the sequence $(b_n)_{n \in \mathbf{N}}$ can only have one prime $p$ such that $b_p \neq 1$ or $0$. we conclude.
\end{proof}

\section {Reference}

[1]{ J.-P. Allouche, Mock characters and the Kronecker symbol. arxiv}

[2] {A. Cobham, “Uniform tag sequences”, Math. Systems Theory 6 (1972), 164–192.}

[3] {Y. Hu, Subword Complexity and (non)-automaticity of certain completely multiplicative functions}

[4] {I.Z. Ruzsa, General multiplicatives functions, ibid. 32 (1977), pp. 313-347.}

[5] {J.-C. Schlage-Puchta, (2011). Completely multiplicative automatic functions. INTEGERS, 11.}

\end{document}